\documentclass{amsart}
\usepackage[dvips]{graphicx,psfrag}
\usepackage{amsmath,amsthm,amssymb,IMjournal}
\usepackage{xcolor}
\newcommand{\R}{\mathbb R}
\newcommand{\N}{\mathbb N}

\newcommand{\Simp}{\mathbb S}

\newcommand{\dd}{\mathrm{d}}
\newcommand{\lam}{\lambda}

\newcommand{\bfx}{\mathbf{x}}

\newcommand{\bfz}{\mathbf{0}}

\newcommand{\hK}{\widehat{K}}

\newcommand{\T}{\mathcal{T}}
\newcommand{\I}{\mathcal{I}}
\newcommand{\PP}{\mathcal{P}}
\newcommand{\Ka}{K_{\alpha}}
\newcommand{\Kab}{K_{\alpha\beta}}

\begin{document}
\newtheorem{theorem}{Theorem}[section]
\newtheorem{definition}[theorem]{Definition}
\newtheorem{lemma}[theorem]{Lemma}
\newtheorem{corollary}[theorem]{Corollary}

\numberwithin{equation}{section}

\title{Extending Babu\v{s}ka-Aziz's theorem to higher-order
Lagrange interpolation}
\author{\|Kenta |Kobayashi|, Kunitachi,
        \|Takuya |Tsuchiya|, Matsuyama}

\rec {September 1, 2015}

\abstract
We consider the error analysis of Lagrange interpolation on triangles and
tetrahedrons.  For Lagrange interpolation of order one, Babu\v{s}ka and Aziz
showed that squeezing a right isosceles triangle perpendicularly does not
deteriorate the optimal approximation order.  We extend their technique
and result to higher-order Lagrange interpolation on both triangles and
tetrahedrons.  To this end, we make use of difference quotients of
functions with two or three variables.  Then, the error estimates on
squeezed triangles and tetrahedrons are proved by a method that is a
straightforward extension of the original given by Babu\v{s}ka-Aziz.
\endabstract



\thanks
The authors are supported by JSPS Grant-in-Aid for Scientific Research
(C) 25400198 and (C) 26400201.
The second author is partially supported by
JSPS Grant-in-Aid for Scientific Research (B) 23340023.
We dedicate this paper to Prof.\ Ivo Babu\v{s}ka in celebration of his
90th birthday.
\endthanks

\newcommand{\vvskip}{\vspace{5pt}}

\section{Introduction}

Lagrange interpolation on triangles and tetrahedrons and the associated error
estimates are important subjects in numerical analysis. In particular, they are
crucial in the error analysis of finite element methods. Let $d = 2$ or $3$.
Throughout this paper, $K \subset \R^d$ denotes a triangle or tetrahedron
with vertices $\bfx_i$, $i=1,\cdots,d+1$.  We always suppose that
triangles and tetrahedrons are closed sets in this paper. Let $\lam_i$
be its barycentric coordinates with respect to $\bfx_i$. By definition,
we have $0 \le \lam_i \le 1$, $\sum_{i=1}^{d+1} \lam_i =1$.  Let $\N_0$
be the set of nonnegative integers, and
$\gamma = (a_1,\cdots,a_{d+1}) \in \N_0^{d+1}$ be a multi-index.
Let $k$ be a positive integer. If $|\gamma| := \sum_{i=1}^{d+1}a_i = k$, then
${\gamma}/{k} := \left({a_1}/{k}, \cdots, {a_{d+1}}/{k}\right)$
can be regarded as a barycentric coordinate in $K$.
The set $\Sigma^k(K)$ of points on $K$ is defined by
\begin{equation*}
   \Sigma^k(K) := \left\{
   \frac{\gamma}{k} \in K \Bigm| |\gamma| = k, \;
     \gamma \in \N_0^{d+1} \right\}.
\end{equation*}

Let $1 < p \le \infty$. From Sobolev's imbedding theorem and Morry's
inequality, we have the continuous imbeddings
\begin{gather*}
   W^{2,p}(K) \subset C^{1,1-d/p}(K), \quad p > d, \\
   W^{2,d}(K) \subset W^{1,q}(K) \subset C^{0,1-d/q}(K),
      \quad \forall q > d, \\
  W^{2,p}(K) \subset W^{1,dp/(d-p)}(K) \subset C^{0,2-d/p}(K),
      \quad \frac{d}{2} < p < d.
\end{gather*}
If $d=3$, we also have the continuous imbeddings
\begin{gather*}
  W^{3,3/2}(K) \subset W^{2,3}(K) \subset W^{1,q}(K)
  \subset C^{0,1-3/q}(K), \quad \forall q > 3, \\
  W^{3,p}(K) \subset W^{2,3p/(3-p)}(K) \subset W^{1,3p/(3-2p)}(K)
  \subset C^{0,3-3/p}(K), \quad 1 < p < \frac{3}{2}.
\end{gather*}
Although Morry's inequality may not be applied, the continuous
imbedding $W^{d,1}(K)$ $\subset C^{0}(K)$ $(d = 2,3)$ still holds.
For the imbedding theorem, see \cite{AdamsFournier},
\cite{Brezis}, and \cite{KJF}.  In the sequel we always suppose that $p$
is taken so that the imbedding $W^{k+1,p}(K) \subset C^{0}(K)$
holds, that is,
\begin{align*}
     1 \le p \le \infty, \quad & \text{ if } d=2, \; k+1 \ge 2 \text{ or }
    d =3, \; k+1 \ge 3, \quad \text{ and } \\
    \frac{3}{2} < p \le \infty, \quad & \text{ if } d=3, \; k+1 =2.
\end{align*}
Note that our discussion includes the case $d=k+1$, $p=1$
that is sometimes ignored in literature.

We define the subset $\T_p^{k}(K) \subset W^{k+1,p}(K)$ by
\begin{gather*}
   \T_p^{k}(K) := \left\{v \in W^{k+1,p}(K) \Bigm|
     v(\bfx) = 0, \; \forall \bfx \in \Sigma^k(K) \right\}.
\end{gather*}
Let $\mathcal{P}_k$ be the set of polynomials with two or three variables
for which the degree is at most $k$.  For a continuous function
$v \in C^0(K)$, the $k$th-order Lagrange interpolation
$\I_K^k v \in \mathcal{P}_k$ is defined by
\[
   v(\bfx) = (\I_K^kv)(\bfx), \qquad \forall \bfx \in \Sigma^k(K).
\]
From this definition, it is clear that
\[
     v - \I_K^k v \in \T_p^{k}(K), \quad \forall 
     v \in W^{k+1,p}(K).
\]

For an integer $m$ such that $0 \le m \le k$, $B_p^{m,k}(K)$
is defined by
\begin{gather*}
    B_p^{m,k}(K) := \sup_{v \in \T_p^{k}(K)}
    \frac{|v|_{m,p,K}}{|v|_{k+1,p,K}}.
\end{gather*}
Note that we have
\begin{gather*}
 B_p^{m,k}(K) = \inf\left\{C ; 
   |v - \I_K^k v|_{m,p,K} \le C |v|_{k+1,p,K}, \
   \forall v \in W^{k+1,p}(K)\right\},
% \label{equivalent}
\end{gather*}
that is, $B_p^{m,k}(K)$ is the \textit{best} constant $C$ for the
error estimation
\[
   |v - \I_K^k v|_{m,p,K} \le C |v|_{k+1,p,K}, \qquad 
    \forall v \in W^{k+1,p}(K).
\]

To establish the mathematical foundation of the finite element methods,
we must show that $B_p^{m,p}(K)$ is bounded.
Many textbooks on finite element methods, such as those by Ciarlet
\cite{Ciarlet}, Brenner-Scott \cite{BrennerScott}, and Ern-Guermond
\cite{ErnGuermond}, present the following theorem.
Let $h_K$ be the diameter of $K$ and  $\rho_K$ be the
radius of the inscribed ball of $K$.

\vvskip
\noindent
\textbf{Shape-regularity}.
{\it Let $\sigma > 2$ be a constant. If $h_K/\rho_K \le \sigma$ and
$h_K \le 1$, then there exists a constant $C = C(\sigma)$ that is
independent of $h_K$ such that}
\[
   \|v - \I_K^1 v\|_{1,2,K} \le C h_K |v|_{2,2,K}, \qquad
   \forall v \in H^2(K).
\]

\vvskip
The maximum of the ratio $h_K/\rho_K$ in a triangulation is called the
\textit{chunkiness parameter} \cite{BrennerScott}.  The shape
regularity, however, is not necessarily needed to obtain an error estimate
for triangles and tetrahedrons.  
For triangles, the following estimations are well-known \cite{BabuskaAziz},
\cite{BarnhillGregory}, \cite{Jamet}.

\vvskip
\noindent
\textbf{The maximum angle condition}. 
{\it Let $\theta_1$, $(\pi/3 \le \theta_1 < \pi)$ be a constant.  If
any angle $\theta$ of $K$ satisfies $\theta \le \theta_1$ and
$h_K \le 1$, then there exists a constant $C = C(\theta_1)$
that is independent of $h_K$ such that}
\begin{equation}
   \|v - \I_K^1 v\|_{1,2,K} \le C h_K |v|_{2,2,K}, \qquad
   \forall v \in H^2(K).
   \label{maximum_angle}
\end{equation}

\vvskip
Later, K\v{r}\'{i}\v{z}ek \cite{Krizek1} introduced the
\textit{semiregularity condition} for triangles, which is equivalent to
the maximum angle condition.  Let $R_K$ be the circumradius of $K$.

\vvskip
\noindent
\textbf{The semiregularity condition}.
{\it Let $p > 1$ and  $\sigma > 0$ be a constant. If $R_K/h_K \le \sigma$ and
$h_K \le 1$, then there exists a constant $C = C(\sigma)$
that is independent of $h_K$ such that}
\begin{equation}
   \|v - \I_K^1 v\|_{1,p,K} \le C h_K |v|_{2,p,K}, \qquad
   \forall v \in W^{2,p}(K).
   \label{krizek} 
\end{equation}

For tetrahedrons, the following estimation is well-known
\cite{Krizek2}, \cite{Duran}.

\vvskip
\noindent
\textbf{K\v{r}\'{i}\v{z}ek's maximum angle condition}. 
{\it Let $\theta_2$, $(\pi/3 \le \theta_2 < \pi)$ be a constant.
Let $\gamma_k$ be the maximum angle of faces of a tetrahedron $K$
and $\varphi_K$ be the maximum angle between faces of $K$.
If  $\gamma_k \le \theta_2$, $\varphi_K \le \theta_2$,
and $h_K \le 1$, then there exists a constant $C = C(\theta_2)$
that is independent of $h_K$ such that}
\begin{equation}
   \|v - \I_K^1 v\|_{1,p,K} \le C h_K |v|_{2,p,K}, \qquad
   \forall v \in W^{2,p}(K), \quad 2 < p \le \infty.
  \label{KrizekDuran}
\end{equation}

Jamet \cite{Jamet} presented 
a general result which covers both triangles and tetrahedrons.
 Let $E_d := \{e_s\}_{s=1}^d \subset \R^d$
be a set of unit vectors which are linearly independent.
Let $\xi \in \R^d$ be an unit vector and
$\theta_s$, $0 \le \theta_s \le \frac{\pi}{2}$ be the angle between
$\xi$ and the line which is defined by $e_s$.  Define
\[
   \theta(E_d) := \max_{\xi \in \R^d} \min_{e_s \in E_d}\{\theta_s\}.
\]
Let $K \subset \R^d$ be $d$-simplex.  Let $N:=d(d+1)/2$ and $E_N$
be the set of $N$ unit vectors that are parallel to the edges of $K$.
Define $\displaystyle \theta_K := \min_{E_d \subset E_N}\{\theta(E_d)\}$.
Note that if $d=2$ and $K$ is an obtuse triangle, then $2\theta_K$ is
the maximum angle of $K$.

\begin{theorem}[Jamet]\label{Jamettheorem}
Let $1 \le p \le \infty$. Let $m \ge 0$, $k \ge 1$ be integers such that
$k+1-m > 2/p$ $(1 < p \le \infty)$ or $k-m \ge 1$ $(p=1)$ if $d=2$, or
$k+1-m > 3/p$ if $d=3$.
Then, the following estimate holds:
\begin{equation}
  |v - \I_K^k v|_{m,p,K} \le C \frac{h_K^{k+1-m}}{(\cos \theta_K)^m}
  |v|_{k+1,p,K},   \quad \forall v \in W^{k+1,p}(K),
  \label{jamet}
\end{equation}
where $C$ depends only on $k$, $p$.
\end{theorem}

\noindent
\textit{Remark}: Note that in \cite[Th\'{e}or\`{e}me~3.1]{Jamet}
the case $d=2$ and $p=1$ is not mentioned explicitly but clearly holds for
triangles.

\vspace{0.2cm}
For further results of error estimations on ``skinny elements'',
readers are referred to the monograph by Apel \cite{Apel}.

The common idea among the above mentined estimations is that (i)
show an error estimate for a particular type of elements,
then (ii) extend it for general elements by affine transformation.
To prove the maximum angle condition for triangles, for example,
Babu\v{s}ka and Aziz showed the following
theorem \cite[Lemma~2.2, Lemma~2.4]{BabuskaAziz}.
Let $\hK$ be the right triangle with vertices
$(0,0)^\top$, $(1,0)^\top$, and $(0,1)^\top$, and $K_\alpha$ be the right
triangle with vertices $(0,0)^\top$, $(1,0)^\top$, and
$(0,\alpha)^\top$ $(0 < \alpha \le 1)$.  That is, $K_\alpha$
is obtained by squeezing $\hK$.

\begin{theorem}[Babu\v{s}ka-Aziz]\label{BabuskaAziz}
There exists a constant independent of $\alpha$ $(0 < \alpha \le 1)$
such that $B_2^{m,1}(K_\alpha) \le C$, $m = 0$, $1$.
As an immediate consequence, we obtain the error estimation of 
Lagrange interpolation $\I_K^1$ on a right triangle $K$; for $m = 0, 1$,
\[
   |v - \I_{K}^1 v|_{m,2,K} \le C h_{K}^{2-m}|v|_{2,2,K}.
\]
\end{theorem}

Theorem~\ref{BabuskaAziz} claims that squeezing a right isosceles
triangle perpendicularly does not deteriorate the 
optimal approximation order of $\I_K^1$.
Babu\v{s}ka and Aziz then claim that the estimate
\eqref{maximum_angle} for general triangular elements
is obtained by affine transformations.
Kobayashi and Tsuchiya \cite{KobayashiTsuchiya1} 
extended Theorem~\ref{BabuskaAziz} for any $p$ $(1 \le p \le \infty)$.

Now, let $\hK$ denote also the reference tetrahedron with
vertices $(0,0,0)^\top$, $(1,0,0)^\top$, $(0,1,0)^\top$, and 
$(0,0,1)^\top$.  Let $\Kab$ be the ``right''
tetrahedron with vertices $(0,0,0)^\top$, $(1,0,0)^\top$, 
$(0,\alpha,0)^\top$, and $(0,0,\beta)^\top$ $(0 < \alpha,\beta \le 1)$.

The aim of this paper is to extend Theorem~\ref{BabuskaAziz} and 
establish the following theorem.

\begin{theorem}\label{squeezingtheorem}
If $d = 2$, there exists a constant $C_{k,m,p}$ such that, 
for $m = 0, \cdots, k$,
\begin{align}
   B_p^{m,k}(K_\alpha) & := \sup_{v \in \T_p^{k}(K_\alpha)}
  \frac{|v|_{m,p,K_\alpha}}{|v|_{k+1,p,K_\alpha}}
   \le C_{k,m,p}, \quad  \; k \ge 1, \; 1 \le p \le \infty.
  \label{extension-2d}
\end{align}
If $d = 3$, there exists a constant $C_{k,m,p}$ such that,
for $m = 0, \cdots, k$,
\begin{align}
   B_p^{m,k}(\Kab) & := \hspace{-0.2cm}
   \sup_{v \in \T_p^{k}(\Kab)}
  \frac{|v|_{m,p,\Kab}}{|v|_{k+1,p,\Kab}}
   \le C_{k,m,p}, \quad
   \begin{cases}
     k - m = 0, & 2 < p \le \infty, \\ 
     k=1, \; m=0, & \frac{3}{2} < p \le \infty, \\
     k\ge2, \; k-m \ge 1, & 1 \le p \le \infty.
   \end{cases}
  \label{extension-3d} 
\end{align}
\end{theorem}

Using Theorem~\ref{squeezingtheorem} and affine transformations,
we can derive an error estimation on general triangles.
See Section~4.

The above mentioned estimations \eqref{maximum_angle},
\eqref{krizek}, \eqref{KrizekDuran}, \eqref{jamet} cover
Theorem~\ref{squeezingtheorem} partially.
We also mention that Shenk \cite{Shenk} showed \eqref{extension-2d} for
$p=2$, $k \ge 1$, $m=0$, $1$, and \eqref{extension-3d} for $p=2$,
$k \ge 2$, $m=0$, $1$.

Because of the restrictions for $m$, $k$, and $p$ in the above mentioned
estimations, it seems that \eqref{extension-2d} with $k=m \ge 2$, $1 \le p \le2$,
and \eqref{extension-3d} with $k=m \ge 2$, $1 \le  p \le 3$ have not yet
been proved.  To prove Theorem~\ref{squeezingtheorem}, we introduce the
difference quotients of functions with two or three variables in
Section~\ref{DQ}.  Then, Theorem~\ref{squeezingtheorem} is proved in
Section~\ref{sect-proof} by a method that is a straightforward extension
of Babu\v{s}la-Aziz's original argument.
  The notations of functional spaces used in
this paper are exactly the same as those in \cite{KobayashiTsuchiya3}.

\section{Difference quotients for multi-variable functions}\label{DQ}
In this section, we define the difference quotients for two- and three-variable
functions.  Our treatment is based on the theory of difference quotients
for one-variable functions given in standard textbooks such as
\cite{Atkinson} and \cite{Yamamoto1}.  All statements in this section
can be readily proved.

For a positive integer $k$, the set $\widehat{\Sigma}^k \subset \hK$  is
defined by
\begin{align*}
   \widehat{\Sigma}^k & := \left\{ \bfx_{\gamma} := \frac{\gamma}{k}
     \in \hK \biggm|  \gamma \in \N_0^d, \; 0 \le |\gamma| \le k \right\},
\end{align*}
where $\gamma/k=(a_1/k,\cdots,a_d/k)$ is understood as the
coordinate of a point in $\widehat{\Sigma}^k$.

For $\bfx_{\gamma} \in \widehat{\Sigma}^k$ and a multi-index $\delta \in \N_0^d$
with $|\gamma| \le k - |\delta|$, we define the correspondence $\Delta^\delta$
between nodes by
\[
   \Delta^\delta \bfx_{\gamma} := \bfx_{\gamma+\delta} \in \widehat{\Sigma}^k.
\]
For two multi-indexes  $\eta=(m_1,\cdots,m_d)$,
$\delta=(n_1,\cdots,n_d)$,  $\eta \le \delta$ means that
$m_i \le n_i$ $(i=1,\cdots,d)$.
Also, $\delta\cdot\eta$ and $\delta !$ are defined by
$\delta\cdot\eta := \sum_{i=1}^d m_i n_i$ and
 $\delta ! := n_1! \cdots n_d!$, respectively.
Using $\Delta^\delta$, we define the \textit{difference quotients}
on $\widehat{\Sigma}^k$ for $f \in C^0(\hK)$ by
\begin{align*}
 f^{|\delta|}[\bfx_{\gamma},\Delta^\delta \bfx_{\gamma}] :=
  k^{|\delta|}\sum_{\eta \le \delta} 
\frac{(-1)^{|\delta|-|\eta|}}{\eta!(\delta-\eta)!}
    f(\Delta^\eta \bfx_{\gamma}).
\end{align*}
Let $\bfz := (0,\cdots,0) \in \N_0^d$.  For simplicity, we denote 
$f^{|\delta|}[\bfx_{\bfz}, \Delta^\delta\bfx_{\bfz}]$ by
$f^{|\delta|}[\Delta^\delta\bfx_{\bfz}]$.  The following are
examples of $f^{|\delta|}[\Delta^\delta\bfx_{\bfz}]$: if $d=2$,
{\allowdisplaybreaks
\begin{align*}
   f^2[\Delta^{(2,0)}\bfx_{(0,0)}] & =
   \frac{k^2}{2} (f(\bfx_{(2,0)}) - 2 f(\bfx_{(1,0)})
     + f(\bfx_{(0,0)})), \\
   f^2[\Delta^{(1,1)}\bfx_{(0,0)}] & =
   k^2 (f(\bfx_{(1,1)}) -  f(\bfx_{(1,0)})
    - f(\bfx_{(0,1)}) + f(\bfx_{(0,0)})), \\
  f^3[\Delta^{(2,1)}\bfx_{(0,0)}] & =
   \frac{k^3}{2} (f(\bfx_{(2,1)}) - 2 f(\bfx_{(1,1)})
   + f(\bfx_{(0,1)}) - f(\bfx_{(2,0)}) \\
 & 
  \hspace{2.42cm}
   + 2 f(\bfx_{(1,0)}) - f(\bfx_{(0,0)})),
\end{align*}
}
and if $d = 3$,
\begin{align*}
  f^4[\Delta^{(2,1,1)}\bfx_{(0,0,0)}] & =
   \frac{k^4}{2} (f(\bfx_{(2,1,1)}) - 2 f(\bfx_{(1,1,1)})
   + f(\bfx_{(0,1,1)}) \\
    & \hspace{0.5cm} 
   - f(\bfx_{(2,0,1)}) + 2 f(\bfx_{(1,0,1)}) - f(\bfx_{(0,0,1)}) \\
    & \hspace{0.5cm} 
   - f(\bfx_{(2,1,0)}) + 2 f(\bfx_{(1,1,0)}) - f(\bfx_{(0,1,0)}) \\
 & \hspace{0.5cm}
   + f(\bfx_{(2,0,0)}) - 2 f(\bfx_{(1,0,0)}) + f(\bfx_{(0,0,0)})).
\end{align*}

Let $\eta \in \N_0^d$ be such that $|\eta|=1$.
The difference quotients clearly satisfy the following
recursive relations:
\begin{align*}
 f^{|\delta|}[\bfx_{\gamma},\Delta^{\delta}\bfx_{\gamma}] & = 
 \frac{k}{\delta\cdot\eta} \left(
 f^{|\delta|-1}[\bfx_{\gamma+\eta},\Delta^{\delta-\eta}\bfx_{\gamma+\eta}]
-  f^{|\delta|-1}[\bfx_{\gamma},\Delta^{\delta-\eta}\bfx_{\gamma}]  \right).
\end{align*}

If $f \in C^k(\hK)$, the difference quotient
$f^{|\delta|}[\bfx_{\gamma},\Delta^\delta \bfx_{\gamma}]$ is written as
an integral of $f$.  Setting $d=2$ and $\delta=(0,s)$, for example,
we have
\begin{gather*}
   f^{1}[\bfx_{(l,q)},\Delta^{(0,1)} \bfx_{(l,q)}]
   = k(f(\bfx_{l,q+1}) - f(\bfx_{lq}))
   = \int_{0}^{1} \partial_{x_2}
    f\left(\frac{l}{k},\frac{q}{k} + \frac{w_1}{k}\right) \dd w_1, \\
 \hspace{-9cm}
f^{s}[\bfx_{(l,p)},\Delta^{(0,s)} \bfx_{(l,q)}]  \\
\hspace{1cm}
= \int_{0}^{1}\int_0^{w_1}\cdots \int_0^{w_{s-1}} \partial^{(0,s)}
  f\left(\frac{l}{k},\frac{q}{k} + \frac{1}{k}(w_1 + \cdots + w_s)
   \right) \dd w_s \cdots \dd w_2 \dd w_1.
\end{gather*}
To provide a concise expression for the above integral, we introduce
the $s$-simplex
\begin{gather*}
   \Simp_s := \left\{(t_1,t_2,\cdots,t_s) \in \R^s \mid 
   t_i \ge 0, \ 0 \le t_1 + \cdots + t_s \le 1 \right\},
\end{gather*}
and the integral of $g \in L^1(\Simp_s)$ on $\Simp_s$ is defined by
\begin{gather*}
  \int_{\Simp_s} g(w_1,\cdots,w_k) \dd\mathbf{W_s}
  := \int_{0}^{1}\int_0^{w_1}\cdots \int_0^{w_{s-1}} 
  g(w_1, \cdots, w_s)  \dd w_s \cdots \dd w_2 \dd w_1.
\end{gather*}
Let us temporarily set $d=2$.
Then, $f^{s}[\bfx_{(l,q)},\Delta^{(0,s)} \bfx_{(l,q)}]$ becomes
\begin{align*}
f^{s}[\bfx_{(l,q)},\Delta^{(0,s)} \bfx_{(l,q)}]
  & = \int_{\Simp_s} \partial^{(0,s)}
  f\left(\frac{l}{k},\frac{q}{k} + \frac{1}{k}(w_1 + \cdots + w_s)
   \right) \dd \mathbf{W_s}.
\end{align*}
For a general multi-index $(t,s)$, we have
\begin{align*}
 f^{t+s}[ & \bfx_{(l,q)},\Delta^{(t,s)} \bfx_{(l,q)}] \\
  & = \int_{\Simp_s}\int_{\Simp_t}\partial^{(t,s)}
  f\left(\frac{l}{k} + \frac{1}{k}(z_1 + \cdots + z_t),
   \frac{q}{k} + \frac{1}{k}(w_1 + \cdots + w_s)
   \right) \dd \mathbf{Z_t} \dd \mathbf{W_s}.
\end{align*}
Let $\square_{\gamma}^\delta$ be the rectangle defined by 
$\bfx_{\gamma}$ and $\Delta^\delta \bfx_{\gamma}$
as the diagonal points.  If $\delta=(t,0)$
or $(0,s)$, $\square_{\gamma}^\delta$ degenerates to a segment.  For
$v \in L^1(\hK)$ and $\square_{\gamma}^\delta$ with $\gamma=(l,q)$, we
denote the integral as
\begin{equation*}
   \int_{\square_{\gamma}^{(t,s)}} v :=
   \int_{\Simp_s}\int_{\Simp_t}
  v\left(\frac{l}{k} + \frac{1}{k}(z_1 + \cdots + z_t),
   \frac{q}{k} + \frac{1}{k}(w_1 + \cdots + w_s)
   \right) \dd \mathbf{Z_t} \dd \mathbf{W_s}.
\end{equation*}
If $\square_{\gamma}^\delta$ degenerates to a segment, the integral is
understood as an integral on the segment.  By this notation, the
difference quotient $f^{t+s}[\bfx_{\gamma},\Delta^{(t,s)} \bfx_{\gamma}]$ is
written as 
\begin{align*}
 f^{t+s}[\bfx_{\gamma},\Delta^{(t,s)} \bfx_{\gamma}] 
  = \int_{\square_{\gamma}^{(t,s)}} \partial^{(t,s)} f.
\end{align*}
Therefore, if $u \in \T_p^k(\hK)$, then we have
\begin{align}
 0 = u^{t+s}[\bfx_{\gamma},\Delta^{(t,s)} \bfx_{\gamma}] 
  = \int_{\square_{\gamma}^{(t,s)}} \partial^{(t,s)} u, \qquad
  \forall \square_{\gamma}^{(t,s)} \subset \hK.
  \label{tomoko}
\end{align}

For the case $d=3$, the integral $\int_{\square_{\gamma}^\delta} v$ is
defined in exactly the same manner.

\section{Proof of Theorem~\ref{squeezingtheorem}}\label{sect-proof}
Let $S \subset \hK$ be a segment.  In the proof of
Theorem~\ref{squeezingtheorem}, the continuity of the trace operator $t$
defined as $t:W^{1,p}(\hK) \ni v \mapsto v|_S \in L^1(S)$ is crucial.
If $d=2$, the continuity of $t$ is standard and is mentioned in many
textbooks such as \cite{Brezis}. For the case $d=3$, the situation
becomes a bit more complicated.  If the continuous inclusion
$W^{k+1,p}(\hK) \subset C^0(\hK)$ holds, the continuity of $t$
is obvious. Even if this is not the case, we still have the following
lemma.  For the proof, see \cite[Theorem~4.12]{AdamsFournier},
\cite[Lemma~2.2]{Duran}, and \cite[Theorem~2.1]{LSU}.

\begin{lemma}\label{lem31}
 Let $d=3$ and $S \subset \hK$ be an arbitrary segment.  Then,
the following trace operators are well-defined and continuous:
\begin{align*}
   t:W^{1,p}(\hK) \to L^p(S), \quad 2 < p < \infty, \qquad
   t:W^{2,p}(\hK) \to L^p(S), \quad 1 \le p < \infty.
\end{align*}
\end{lemma}

For a multi-index $\delta$, $|\delta| \ge 1$, $p$ is taken so that
\begin{align}
 \begin{cases}
  2 < p \le \infty, & \text{ if } k+1-|\delta| = 1, \; d = 3,\\
 1 \le p \le \infty, &
   \text{ if } k+1-|\delta| \ge 2, \; d=3 \text{ or }
  d= 2. 
  \end{cases}
  \label{p-choice}
\end{align}
The set $\Xi_p^{\delta,k} \subset W^{k+1-|\delta|,p}(\hK)$ is then defined by
\begin{align*}
  \Xi_p^{\delta,k} & := \left\{ v \in W^{k+1-|\delta|,p}(\hK) \Bigm| 
    \int_{\square_{lp}^{\delta}} v = 0,\quad \forall
   \square_{lp}^{\delta} \subset \hK \right\}.
\end{align*}
By Lemma~\ref{lem31} and \eqref{p-choice}, $\Xi_p^{\delta,k}$ is
well-defined. Note that $u \in \T_p^{k}(\hK)$
implies $\partial^\delta u \in \Xi_p^{\delta,k}$ by definition and
\eqref{tomoko}.

\begin{lemma}\label{lem32} 
We have $\Xi_p^{\delta,k} \cap \mathcal{P}_{k-|\delta|} = \{0\}$.
That is, if $q \in \mathcal{P}_{k-|\delta|}$ belongs to $\Xi_p^{\delta,k}$,
then $q=0$.
\end{lemma}
\begin{proof} We notice that
$\mathrm{dim} \mathcal{P}_{k-|\delta|} = \#\{\square_{lp}^\delta \subset \hK \}$.
For example, if $k=4$, $d=2$, and $|\delta|=2$, then
$\mathrm{dim}\mathcal{P}_2 = 6$.  This corresponds to the fact that,
in $\hK$, there are six squares with size $1/4$  for $\delta=(1,1)$ and
there are six horizontal segments of length $1/2$ for $\delta=(2,0)$.
All their vertices (corners and end-points) belong to $\Sigma^4(\hK)$
(see Figure~1).  The situation is the same for $d=3$.
Now, suppose that $v \in \mathcal{P}_{k-|\delta|}$ satisfies
$\int_{\square_{lp}^{\delta}} q = 0$ for all
$\square_{lp}^{\delta} \subset \hK$.  This condition is linearly
independent and determines $q = 0$ uniquely.
\end{proof}

\begin{figure}[thb]
\begin{center}
  \includegraphics[width=6.3cm]{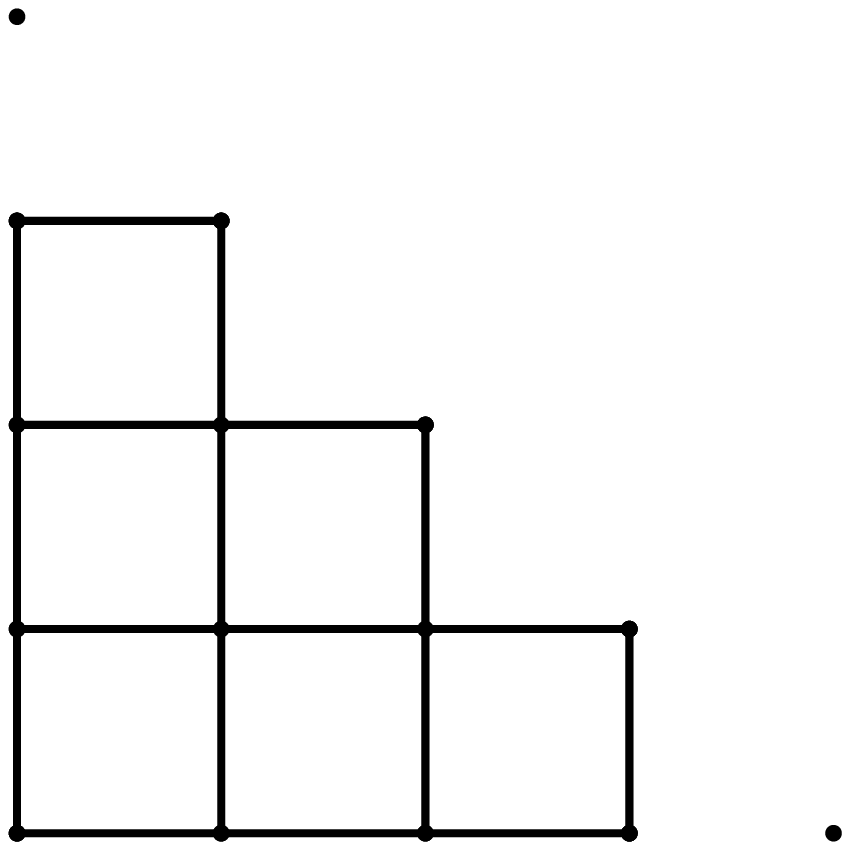} 
  \includegraphics[width=6.3cm]{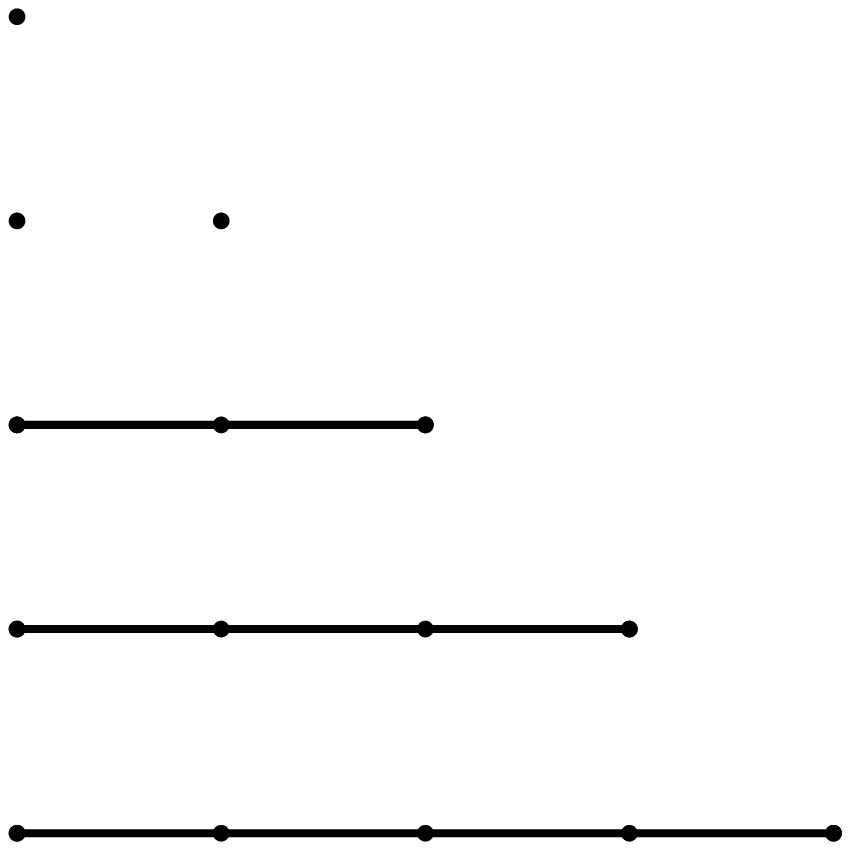} 
\caption{The six squares of size $1/4$ for $\delta=(1,1)$
and the (union of) six segments of length $1/2$ for
$\delta=(2,0)$ in $\hK$.}
 \label{squeezed}
\end{center}
\end{figure}

The constant $A_p^{\delta,k}$ is defined by
\begin{align*}
   A_p^{\delta,k} := \sup_{v \in \Xi_p^{\delta,k}} \frac{|v|_{0,p,\hK}}
  {|v|_{k+1-|\delta|,p,\hK}}.
\end{align*}
The following lemma is an extension of \cite[Lemma~2.1]{BabuskaAziz}.

\begin{lemma}\label{lem33}
Let $p$ be given by \eqref{p-choice}. Then, we have
$A_p^{\delta,k} < \infty$.
\end{lemma}
\proof The proof is by contradiction.  Suppose that
$A_p^{\delta,k} = \infty$.  Then there exists a sequence 
$\{w_k\}_{i=1}^\infty \subset \Xi_p^{\delta,k}$ such that
$|w_n|_{0,p,\hK} = 1$ and 
$\lim_{n \to \infty} |w_n|_{k+1-|\delta|,p,\hK} = 0$.
By \cite[Theorem~3.1.1]{Ciarlet}, there exists 
$\{q_n\}\subset \PP_{k-|\delta|}$ such that
\begin{align*}
\|w_n + q_n\|_{k+1-|\delta|,p,\hK} 
   \le \inf_{q \in \PP_{k-|\delta|}}
  \|w_n + q\|_{k+1-|\delta|,p,\hK} + \frac{1}{n}
   \le C |w_n|_{k+1-|\delta|,p,\hK} + \frac{1}{n}
\end{align*}
and $\displaystyle\lim_{n \to \infty} \|w_n + q_n\|_{k+1-|\delta|,p,\hK} = 0$.
Because $\{w_n\}\subset W^{k+1-|\delta|,p}(\hK)$ is bounded,
$\{q_n\} \subset \PP_{k-|\delta|}$ is bounded as well.
Hence, there exists a subsequence $\{q_{n_i}\}$ such that
$q_{n_i}$ converges to $\bar{q} \in \PP_{k-|\delta|}$ and
$\lim_{n_i \to \infty} \|w_{n_i} + \bar{q}\|_{k+1-|\delta|,p,\hK} = 0$.
By definition and the continuity of the trace operator,
we have $\int_{\square_{lp}^\delta} w_{n_i} = 0$ and
\[
  0 = \lim_{n_i\to\infty} \int_{\square_{lp}^\delta}
   (w_{n_i} + \bar{q}) = \int_{\square_{lp}^\delta} \bar{q}, \qquad
    \forall \square_{lp}^\delta \subset \hK.
\]
Therefore, it follows from Lemma~\ref{lem31} that $\bar{q} = 0$.
This implies that
\[
  0 = \lim_{n_i \to \infty}\|w_{n_i}\|_{k+1-|\delta|,p,\hK} 
  \ge \lim_{n_i \to \infty}|w_{n_i}|_{0,p,\hK} = 1,
\]
which is a contradiction.
\endproof

\proof[The proof of Theorem~\ref{squeezingtheorem}]
The proof is an direct extension of the proof given in 
\cite[Lemma~2.2]{BabuskaAziz}.  Let $d=2$ initially. Define the linear
transformation $F_{\alpha}:\R^2 \to \R^2$ by
\[
   (x^*, y^*)^\top = (x, \alpha y)^\top, \qquad
   (x,y)^\top \in \R^2, \quad 0 < \alpha \le 1,
\]
which squeezes the reference element $\hK$ perpendicularly to
$\Ka:=F_\alpha(\hK)$.  Take an arbitrary $v \in W^{k+1,p}(\Ka)$ and define
$u \in W^{k+1,p}(\hK)$ by $u(x,y):= v(x,\alpha y)$.
To make the formula concise, we introduce the following notation.
For a multi-index $\gamma = (a, b) \in \N_0^2$ and a real $t \neq 0$,
$(\alpha)^{\gamma t} := \alpha^{b t}$.  Let $1 \le p < \infty$ and
 $1 \le m \le k$.  Because $u \in \T_p^k(\hK)$
and $\partial^\delta u \in \Xi_p^{\delta,k}$,
we may apply Lemma~\ref{lem33} and obtain
{\allowdisplaybreaks
\begin{align}
  \frac{|v|_{m,p,\Ka}^p}{|v|_{k+1,p,\Ka}^p}
   &  = \frac{\displaystyle\sum_{|\gamma| = m} \frac{m!}{\gamma!}
      (\alpha)^{-\gamma p}\left|\partial^{\gamma}u\right|_{0,p,\hK}^p}
      {\displaystyle \sum_{|\delta|= k+1} \frac{(k+1)!}{\delta!}
     (\alpha)^{- \delta p} \left|\partial^{\delta} u  \right|_{0,p,\hK}^p}
     \label{BabuskaAziztechnique} \\
  & = \frac{\displaystyle \sum_{|\gamma| = m}
    \frac{m!}{\gamma!}(\alpha)^{-\gamma p}
         \left|\partial^{\gamma}u \right|_{0,p,\hK}^p }
       {\displaystyle \sum_{|\gamma| = m} \frac{m!}{\gamma!}
      (\alpha)^{-\gamma p}
       \left(\sum_{|\eta| = k+1-m} \frac{(k+1-m)!}
             {\eta!(\alpha)^{\eta p} }
         \left|\partial^{\eta}
                 (\partial^{\gamma}u)\right|_{0,p,\hK}^p\right)} \notag \\
  & \le \frac{\displaystyle \sum_{|\gamma| = m}
    \frac{m!}{\gamma!}(\alpha)^{- \gamma p}
         \left|\partial^{\gamma}u \right|_{0,p,\hK}^p }
       {\displaystyle \sum_{|\gamma| = m} \frac{m!}{\gamma!}
      (\alpha)^{-\gamma p}
       \left(\sum_{|\eta| = k+1-m} \frac{(k+1-m)!}{\eta!}
         \left|\partial^{\eta}
                 (\partial^{\gamma}u)\right|_{0,p,\hK}^p\right)} \notag \\
   & = \frac{\sum_{|\gamma| = m} \frac{m!}{\gamma!}(\alpha)^{- \gamma p}
         \left|\partial^{\gamma}u \right|_{0,p,\hK}^p}
       {\sum_{|\gamma| = m} \frac{m!}{\gamma!}
      (\alpha)^{-\gamma p}
        \left|\partial^{\gamma}u \right|_{k+1-m,p,\hK}^p} \notag \\
   & \le \frac{\sum_{|\gamma|=m} 
      \frac{m!}{\gamma!}(\alpha)^{- \gamma p}
              |\partial^{\gamma}u|_{0,p,\hK}^p }
       {\sum_{|\gamma|=m} \frac{m!}{\gamma!}(\alpha)^{- \gamma p}
         \left(A_p^{\gamma,k}\right)^{-1}
        |\partial^{\gamma}u|_{0,p,\hK}^p} \le C_{k,m,p}^p, \notag
\end{align}
}
where $C_{k,m,,p} := \max_{|\gamma|=m} A_p^{\gamma,k}$.
Here, we use the equality
\[
  \frac{(k+1)!}{\delta!} =
   \sum_{\substack{\gamma+\eta = \delta \\ |\gamma|=m, |\eta|=k+1-m}}
   \frac{m!}{\gamma!} \frac{(k+1-m)!}{\eta!}.
\]
Hence, we obtain \eqref{extension-2d} for this case.  If $m=0$, we have
\begin{align}
  \frac{|v|_{0,p,\Ka}^p}{|v|_{k+1,p,\Ka}^p}
   &  = \frac{|u|_{0,p,\hK}^p}
      {\displaystyle \sum_{|\delta|= k+1} \frac{(k+1)!}{\delta!}
     (\alpha)^{- \delta p} \left|\partial^{\delta} u  \right|_{0,p,\hK}^p}
     \label{BabuskaAziztechnique-2} \\
  & \le \frac{|u|_{0,p,\hK}^p}
      {\displaystyle \sum_{|\delta|= k+1} \frac{(k+1)!}{\delta!}
      \left|\partial^{\delta} u  \right|_{0,p,\hK}^p}
   = \frac{|u|_{0,p,\hK}^p}{|u|_{k+1,p,\hK}^p} 
   \le B_k^{0,p}(\hK)^p < \infty. \notag
\end{align}
Setting $p=\infty$ and $1 \le m \le k$, we have
{\allowdisplaybreaks
\begin{align}
  \frac{|v|_{m,\infty,\Ka}}{|v|_{k+1,\infty,\Ka}}
   &  = \frac{ \displaystyle \max_{|\gamma|=m} \left\{(\alpha)^{-\gamma}
  \left|\partial^{\gamma}u\right|_{0,\infty,\hK}\right\}}
      {\displaystyle \max_{|\delta|= k+1} \left\{ (\alpha)^{-\delta} 
     \left|\partial^{\delta} u\right|_{0,\infty,\hK}\right\}}
     \label{BabuskaAziztechnique-3} \\
  & = \frac{ \displaystyle \max_{|\gamma|=m} \left\{(\alpha)^{-\gamma}
      \left|\partial^{\gamma}u\right|_{0,\infty,\hK}\right\}}
       {\displaystyle \max_{|\gamma|= m} \left\{(\alpha)^{-\gamma} 
     \max_{|\eta|=k+1-m} \left\{ (\alpha)^{-\eta}
     \left|\partial^{\eta}(\partial^\gamma u)\right|_{0,\infty,\hK}
      \right\}\right\}} \notag \\
  & \le \frac{ \displaystyle \max_{|\gamma|=m} \left\{(\alpha)^{-\gamma}
      \left|\partial^{\gamma}u\right|_{0,\infty,\hK}\right\}}
       {\displaystyle \max_{|\gamma|= m} \left\{(\alpha)^{-\gamma} 
     \max_{|\eta|=k+1-m} \left\{
     \left|\partial^{\eta}(\partial^\gamma u)\right|_{0,\infty,\hK}
        \right\}\right\}} \notag \\
   & = \frac{\max_{|\gamma|=m} \left\{(\alpha)^{-\gamma}
      \left|\partial^{\gamma}u\right|_{0,\infty,\hK}\right\}}
       {\max_{|\gamma|= m} \left\{(\alpha)^{-\gamma} 
     \left|\partial^\gamma u\right|_{k+1-m,\infty,\hK}
        \right\}} \notag \\
     & \le \frac{\max_{|\gamma|=m} \left\{(\alpha)^{-\gamma}
      \left|\partial^{\gamma}u\right|_{0,\infty,\hK}\right\}}
       {\max_{|\gamma|= m} \left\{(\alpha)^{-\gamma} 
        \left(A_\infty^{\gamma,k}\right)^{-1}
     \left|\partial^\gamma u\right|_{0,\infty,\hK}
        \right\}} \le C_{k,m,\infty}, \notag
 %       |\partial^{\gamma}u|_{0,p,\hK}^p} \le C_{k,m,p}^p, \notag
\end{align}
}
where $C_{k,m,\infty} := \max_{|\gamma|=m} A_\infty^{\gamma,k}$.
Now, the case with $p = \infty$ and $m=0$ is obvious.
Therefore, \eqref{extension-2d} is proved. 

Next, let $d=3$ and repeat the above proof.
Define the linear transformation
$F_{\alpha\beta}:\R^3 \to \R^3$ by
\[
   (x^*, y^*,z^*)^\top = (x, \alpha y, \beta z)^\top, \qquad
   (x,y,z)^\top \in \R^3, \quad 0 < \alpha,\beta \le 1,
\]
which squeezes the reference tetrahedron $\hK$ perpendicularly to
$\Kab:=F_{\alpha\beta}(\hK)$.  Take an arbitrary $v \in W^{k+1,p}(\Kab)$
and define $u \in W^{k+1,p}(\hK)$ by $u(x,y,z):= v(x,\alpha y,\beta z)$.
Let $p$ be given by \eqref{p-choice} with $m=|\delta|$.  To make formula
concise, we introduce the following notation. 
For a multi-index $\gamma = (a, b, c) \in \N_0^3$ and a real
$t \neq 0$, $(\alpha,\beta)^{\gamma t} := \alpha^{b t}\beta^{c t}$.
Because $u \in \T_p^k(\hK)$ and $\partial^\delta u \in \Xi_p^{\delta,k}$,
we may apply Lemma~\ref{lem33} as above.  Thus, we may repeat
\eqref{BabuskaAziztechnique}, \eqref{BabuskaAziztechnique-2}, and
\eqref{BabuskaAziztechnique-3} by replacing $\Ka$ with $\Kab$,
$(\alpha)^{\gamma p}$ with $(\alpha,\beta)^{\gamma p}$, etc.
Thus, \eqref{extension-3d} is proved. 
\endproof

\section{Concluding Remarks}
Theorem~\ref{squeezingtheorem} deals with only right triangles
and ``right'' tetrahedrons.  Based on
Theorem~\ref{squeezingtheorem}, a new error estimation of Lagrange
interpolation on triangles is obtained in \cite{KobayashiTsuchiya3}.
It should be emphasized that no geometric condition on triangles
is imposed in Theorem~\ref{thm4.1}. 

\begin{theorem}[Kobayashi-Tsuchiya \cite{KobayashiTsuchiya3}]
\label{thm4.1}
Let $K$ be an arbitrary triangle.  Let
$1 \le p \le \infty$, and $k$, $m$ be integers such that $k \ge 1$ and
$0 \le  m \le k$.  Then, for the $k$th-order Lagrange interpolation
 $\I_K^k$ on $K$, the following estimation holds:
\begin{equation}
   |v - \I_K^k v|_{m,p,K} \le C
   \left(\frac{R_K}{h_K}\right)^m h_K^{k+1-m} |v|_{k+1,p,K}
  =  C R_K^m h_K^{k+1-2m} |v|_{k+1,p,K}
  \label{main-estimate}
\end{equation}
for any $v \in W^{k+1,p}(K)$,
where the constant $C$ depends only on $k$, $p$ and is independent of
the geometry of $K$.
\end{theorem}

Any tetrahedron can be obtained from a ``skinny right'' tetrahedron
by an affine transformation.  To obtain an error estimate, we need
to estimate the ratio of the maximum and minimum singular values of
the Jacobian matrix of the affine transformation.
 If we obtain an expression of the ratio in terms of
geometric quantities of the tetrahedron, a new error estimation would
be obtained.  The authors hope that they will report further development
of error estimations on tetrahedrons in near future.

{\small
{\em Kenta Kobayashi}, 
          Graduate School of Commerce and Management, 
       Hitotsubashi University, Kunitachi, 186-8601, Japan,
       e-mail:  \texttt{kenta.k@r.hit-u.ac.jp}. \\
{\em Takuya Tsuchiya},
       Graduate School of Science and Engineering,
          Ehime Univesity, Matsuyama, 790-8577, Japan,
          e-mail: \texttt{tsuchiya@math.sci.ehime-u.ac.jp}.
}
\end{document}